\documentclass[hidelinks,11 pt]{article} 
\usepackage[top=22mm, bottom=22mm, left=22mm, right=22mm]{geometry}
\usepackage{rotating}
\usepackage{amsmath,amsfonts,amssymb,amsthm,enumerate,hyperref,multicol, graphicx, array, caption, subcaption, color,makecell}
\hypersetup{colorlinks,citecolor=blue, linkcolor=blue}
\usepackage{booktabs,siunitx}
\usepackage[dvipsnames]{xcolor}
\usepackage{multirow}
\usepackage{mathtools}
\usepackage{bigstrut}
\usepackage{float}
\usepackage{tabularx}
\usepackage{pgfplots}
\allowdisplaybreaks
\usepackage{titlesec}
\usepackage{enumitem}
\usepackage{authblk}

\usepackage{lipsum}
\definecolor{orcidlogocol}{HTML}{A6CE39}
\usepackage{orcidlink}
\newtheorem{theorem}{Theorem}
\newtheorem{definition}[theorem]{Definition}

\newtheorem{example}[theorem]{Example}

\newtheorem{lemma}{Lemma}[theorem]
\newtheorem{rmrk}[theorem]{Remark}
\numberwithin{equation}{section}
\usepackage{orcidlink}
\numberwithin{equation}{section}
\numberwithin{lemma}{section}
\numberwithin{theorem}{section}
\numberwithin{corollary}{section}
\begin{document}
	{	{ \LARGE{\title{{\bf Convergence analysis for pseudo-monotone variational inequality problem involving projections onto a moving ball}}}}
		\author[1]{Watanjeet Singh$^{\orcidlink{0000-0002-6456-357X}}$\thanks{Corresponding author: watanjeetsingh@gmail.com }}
		\author[1]{Sumit Chandok$^{\orcidlink{0000-0003-1928-2952}}$ }
		\affil[1]{\small{Department of Mathematics, Thapar Institute of Engineering and Technology, \protect\\ Patiala, 147001, Punjab, India\protect\\
				{watanjeetsingh@gmail.com}, {sumit.chandok@thapar.edu}}}
		\date{}
		\maketitle
		\begin{abstract}
			This paper presents an iterative scheme that converges to the solution of a pseudo-monotone variational inequality problem in the setting of $\mathbb{R}^{n}$. Traditional methods often require projections onto the feasible set $\mathfrak{C}$ or onto a half-space containing $\mathfrak{C}$. However, computing projections onto a complicated feasible set can be difficult, and projections onto a half-space may fall outside $\mathfrak{C}$. Keeping this in mind, we aim to develop an iterative scheme that projects onto a ball that is contained in a feasible set and has an explicit expression. Our iterative scheme does not require prior knowledge of the Lipschitz constant of the cost operator. Finally, we provide some numerical experiments to show the effectiveness of our algorithm.
		\end{abstract}
		\section{Introduction}
		The theory of variational inequality problems, in short, VIP, is an important mathematical problem that has been widely recognized for solving numerous problems arising in the real world. The origins of the theory trace back to 1964, when Stampacchia \cite{1a} introduced it to deal with the partial differential equations simulated from mechanics. The classical problem is to find $x \in \mathfrak{C}$ such that
		\begin{align} \label{vip1}
			\langle \mathcal{A}(x),y-x\rangle \geq 0,
		\end{align}
		for all $y \in \mathfrak{C}$, where $\mathfrak{C}$ is non-empty closed and convex subset of $\mathbb{R}^{n}$, and $\mathcal{A}: \mathbb{R}^{n} \to \mathbb{R}^{n}$ is a given mapping. The solution set to the problem \eqref{vip1} is denoted by $VI(\mathfrak{C},\mathcal{A})$.
		\\
		The projected gradient method \cite{1} is one of the first methods developed to find a solution for VIP. This method uses only one projection onto the feasible set $\mathfrak{C}$. However, the convergence is guaranteed only under the assumption that the cost operator involved is strongly monotone and Lipschitz continuous. This stronger condition on the cost operator $\mathcal{A}$ limits the scope of applications of this method. To relax this condition, Korpelevich \cite{2} proposed the extragradient method in the setting of finite-dimensional Euclidean space. The cost operator is assumed to be monotone and Lipschitz continuous. It is worth noting that, although the condition on $\mathcal{A}$ has been relaxed, this has been compensated by increasing the number of projections onto the feasible set in each iteration. In general, computing projections onto an arbitrary closed and convex set is a complicated task. This limitation can seriously affect the effectiveness of the extragradient method. Thus, a natural question arises whether we can reduce the number of projections onto the feasible set $\mathfrak{C}$ in each iteration. Censor et al. \cite{3} proposed a subgradient extragradient method to reduce the number of projections onto $\mathfrak{C}$ in each iteration. The method proposed by Censor et al. \cite{3} requires one projection onto $\mathfrak{C}$ followed by a projection onto a half-space. This is effective approach as the projection onto the half-space has an explicit expression. They obtained a weak convergence result assuming that the cost operator $\mathcal{A}$ is monotone and $L$-Lipschitz continuous. 
		\\
		It is known that any closed and convex set $\mathfrak{C}$ can be expressed as 
		\begin{align}
			\mathfrak{C}=\{x \in \mathbb{R}^{n}: f(x)\leq 0\},
		\end{align}
		where $f:\mathbb{R}^{n} \to \mathbb{R}$ is convex function. For instance, we can consider $f(x) \coloneqq dist(x,\mathfrak{C})$, where "dist" is the distance function. Notice that all the algorithms discussed rely on computing projections onto the feasible set. However, such projections are not always straightforward to compute and can often be computationally expensive. 
		\\
		In 2019, Cao and Guo \cite{4} calculated two projections onto a half-space that contains $\mathfrak{C}$ and provided a weak convergence result using a fixed step size. However, this step size depends on the Lipschitz constant of the cost operator as well as the Gâteaux differential of the convex function $f(x)$, which may limit the practical applicability of their method. Ma and Wang \cite{Wang} improved \cite{4} by proposing a new algorithm that uses a self-adaptive step size which do not require prior knowledge of the Lipschitz constant of the cost operator. Nevertheless, their method still requires computing the Lipschitz constant corresponding to the Gâteaux differential of the convex function $f(x)$. In 2022, Chen and Ye \cite{5} proposed a modified subgradient extragradient method in the setting of Hilbert space. The projections are calculated onto two different half-spaces in each iteration. The convergence is proved under the assumption that cost operator is monotone, $L$-Lipschitz continuous and the following Condition \textbf{(T)} holds:
		\\
		\textbf{(T):} There exists a constant $K>0$ such that $\|\mathcal{A}(x)\| \leq K \|\nabla f(x)\|$, for any $x \in \partial \mathfrak{C}$, where $\partial \mathfrak{C}$ is the boundary of $\mathfrak{C}$. \\
		In order to remove Condition \textbf{(T)}, Zhang et al. \cite{6} proposed a new low-cost feasible projection algorithm in which each projection is calculated onto a ball contained in the set $\mathfrak{C}$. The strong convergence is carried out with the cost operator being pseudo-monotone, $L$-Lipschitz continuous and Slater's condition holds on $\mathfrak{C}$, i.e., $\{x \in \mathbb{R}^{n}: f(x)<0\} \neq 0$. Recently, Feng et al. \cite{7} proposed another variant of the moving ball method using a fixed step size that depends on the Lipschitz constant of the cost operator.
		\\
		It is worth noting that both Zhang et al. \cite{6} and Feng et al. \cite{7} methods used a step size that requires the value of the Lipschitz constant of the cost operator. Taking inspiration from the above works, we present an iterative scheme that calculates projection onto a ball that is contained in the set $\mathfrak{C}$ by combining the extragradient method and projection-contraction method. By incorporating a line search rule that eliminates the need for prior knowledge of the Lipschitz constant, we establish a convergence result for variational inequalities involving a pseudo-monotone operator.
		\\
		This paper is organized as follows: In section 2, we recall some definitions and lemmas required to understand our main result. In section 3, we present our iterative algorithm and prove its convergence result under some assumptions. Finally, we provide some numerical experiments to discuss the performance of our algorithm in comparison with other existing algorithms.

		\section{Preliminaries}
		In this section, we recall some of the important definitions and results that are required to understand our main result. For any sequence $\{x_{n}\}$, the weak and strong convergence to an element $x$ is denoted by $x_{n} \rightharpoonup x$ and $x_{n} \to x$, respectively, as $n \to \infty$. For any $x,y \in \mathbb{R}^{n}$, and $\alpha \in \mathbb{R}$, we have
		\begin{align*}
			\|x+y\|^2 &\leq \|x\|^2+2\langle y,x+y \rangle,\\
			\|x+y\|^2 & = \|x\|^2+\|y\|^2+2\langle x,y \rangle,\\
			\|\alpha x +(1-\alpha)y\|^2&=\alpha \|x\|^2+(1-\alpha)\|y\|^2-\alpha (1-\alpha)\|x-y\|^2.
		\end{align*}
		\begin{definition}
			Let $\mathcal{A}:\mathbb{R}^{n} \to \mathbb{R}^{n}$ be a mapping, then A is said to be:
			\begin{enumerate}
				\item[(i)] $L$-Lipschitz continuous if there a constant $L>0$ such that $\|\mathcal{A}(x)-\mathcal{A}(y)\| \leq L\|x-y\|$, for all $x,y \in \mathbb{R}^{n}$.
				\item[(ii)] monotone if $\langle \mathcal{A}(x)-\mathcal{A}(y),x-y \rangle \geq 0$, for all $x,y\in \mathbb{R}^{n}$.
				\item[(iii)]  pseudo-monotone if $\langle \mathcal{A}(x),y-x \rangle \geq 0$ implies $\langle \mathcal{A}(y),y-x \rangle \geq 0$, for all $x,y \in \mathbb{R}^{n}$.
			\end{enumerate}
		\end{definition}
		For each point $x \in \mathbb{R}^{n}$, there is a unique nearest point $P_{\mathfrak{C}}(x)$ in $\mathfrak{C} \subset \mathbb{R}^{n}$, such that $$P_{\mathfrak{C}}(x)=\arg \min\{\|x-y\|, y\in \mathfrak{C}\}.$$ 
		The mapping $P_{\mathfrak{C}}:\mathbb{R}^{n} \to \mathfrak{C}$ is called the metric projection of $\mathbb{R}^{n}$ onto $\mathfrak{C}$. It is known that the metric projection $P_{\mathfrak{C}}$ is non-expansive and satisfy the following result:
		\begin{lemma} \label{le1}
			Let $\mathfrak{C}$ be a non-empty closed convex subset of $\mathbb{R}^{n}$ and $P_{\mathfrak{C}}$ be the metric projection from $\mathbb{R}^{n}$ to $\mathfrak{C}$. Then the following results hold for any $x \in \mathbb{R}^{n}$:
			\begin{enumerate}
				\item[(i)] $\langle x-P_{\mathfrak{C}}(x),P_{\mathfrak{C}}(x)-y \rangle \geq 0$, for all $y \in \mathfrak{C}$. 
				\item[(ii)] $\|P_{\mathfrak{C}}(x)-y\|^2 \leq \|x-y\|^2-\|x-P_{\mathfrak{C}}(x)\|^2$, for all $y \in \mathfrak{C}$.
			\end{enumerate}
		\end{lemma}
		The projection onto a ball has an explicit expression, as described in the following result:
		\begin{lemma} \cite{8}\label{baush1}
			Let $\mathfrak{D}$ to be a nonempty closed convex subset of $\mathbb{R}^{n}$, and $\mathfrak{G}=\mathfrak{D}+\mathfrak{B}(0;r)$, where $\mathfrak{B}(0;r)$ is a closed ball centered at origin with radius $r>0$. Then, $\mathfrak{G}$ is nonempty closed convex subset of $\mathbb{R}^{n}$ and for any $x \in \mathbb{R}^{n}$, the projection $P_{\mathfrak{G}}(x)$ is defined as:
			\[
			P_{\mathfrak{G}}(x) =
			\begin{cases}
				x, & \text{if } \|x-P_{\mathfrak{D}}(x)\| \leq r; \\
				P_{\mathfrak{D}}(x)+r\frac{x-P_{\mathfrak{D}}(x)}{\|x-P_{\mathfrak{D}}(x)\|}, & \text{otherwise }.
			\end{cases}
			\]
		\end{lemma}
		The following inequality is useful for functions whose gradients are Lipschitz continuous:
		\begin{lemma} \cite{9}
			If $f:\mathbb{R}^{n} \to \mathbb{R}$ is continuously differentiable function and $\nabla f $ is Lipschitz continuous with constant $L_{f}$, then we have
			\begin{align*}
				f(y) \leq f(x)+\langle \nabla f(x),y-x\rangle+\frac{L_{f}}{2}\|y-x\|^2, ~\text{for all}~x,y \in \mathbb{R}^{n}.
			\end{align*}
		\end{lemma}
		For any continuously differentiable convex function $f:\mathbb{R}^{n} \to \mathbb{R}$ with a $L_{f}$-Lipschitz continuous gradient and $x \in \mathbb{R}^{n}$, we define the moving ball $\mathfrak{B}(x)$ as follows:
		\begin{align} \label{ball1}
			\mathfrak{B}(x)&\coloneqq \{y \in \mathbb{R}^{n}: f(x)+\langle \nabla f(x),y-x\rangle+\frac{L_{f}}{2}\|y-x\|^2 \leq 0\} \notag \\
			&=\{y \in \mathbb{R}^{n}:\|y-(x-\frac{1}{L_{f}}\nabla f(x))\|^2\leq \frac{1}{L_{f}^2}\|\nabla f(x)\|^2-\frac{2}{L_{f}}f(x)\}.
		\end{align}
		Now, we give a result based on the properties of $\mathfrak{B}(x)$:
		\begin{lemma} \cite{10}
			Let $\mathfrak{C}=\{x \in \mathbb{R}^{n}:f(x) \leq 0\}$, where $f:\mathbb{R}^{n} \to \mathbb{R}$ is a continuously differentiable convex function. Suppose Slater's condition is satisfied for $\mathfrak{C}$, and let $\mathfrak{B}(x)$ be defined as \eqref{ball1}. Then for any $x \in \mathfrak{C}$, we have
			\begin{enumerate}
				\item $\mathfrak{B}(x)$ is nonempty closed convex set;
				\item $\mathfrak{B}(x) \subseteq \mathfrak{C}$;
				\item Slater's condition is satisfied for $\mathfrak{B}(x)$.
			\end{enumerate}
		\end{lemma}
		\begin{lemma} \cite{11} \label{2.5lem}
			Assume that the solution set $VI(\mathfrak{C},\mathcal{A})$ is nonempty, and $\mathfrak{C}=\{x \in \mathbb{R}^{n}:f(x) \leq 0\}$, where $f:\mathbb{R}^{n} \to \mathbb{R}$ is a continuously differentiable convex function. Then $x^{*} \in VI(\mathfrak{C},\mathcal{A})$ if and only if either
			\begin{enumerate}
				\item $\mathcal{A}(x^{*})=0,$ or
				\item $x^{*} \in \partial \mathfrak{C} $ and there exists $\eta >0$ such that $\mathcal{A}(x^{*})=-\eta \nabla f(x^{*})$.
			\end{enumerate}
		\end{lemma}

		\section{Main Results}
		In this section, we present our main algorithm and results. Suppose that the following assumptions hold:
		\begin{enumerate}
			\item[(A1)] The set $\mathfrak{C}$ is defined by 
			\begin{align*}
				\mathfrak{C}=\{x \in \mathbb{R}^{n}:f(x) \leq 0\};
			\end{align*}
			where $f:\mathbb{R}^{n} \to \mathbb{R}$ is a continuously differentiable convex function and $\nabla f$ is $L_{f}$-Lipschitz continuous on $\mathbb{R}^{n}$.
			\item [(A2)] $\mathcal{A}: \mathbb{R}^{n} \to \mathbb{R}^{n}$ is pseudo-monotone and $L$-Lipschitz continuous on $\mathfrak{C}$.
			\item [(A3)] The solution set $VI(\mathfrak{C},\mathcal{A})$ is non-empty.
			\item [(A4)] Slater's condition is satisfied for $\mathfrak{C}$; i.e. there exists $\hat{x} \in \mathbb{R}^{n}$ such that $f(\hat{x})<0$.
		\end{enumerate}
		We now present our main algorithm.
		
		\begin{table}[ht]
			\centering
			\begin{tabular}{c}
				\hline
				\textbf{Algorithm 3.1: New low-cost iterative moving ball algorithm }  \\
				\hline
			\end{tabular}
		\end{table} 
		\textbf{Initialization:} Given $\mu,\delta \in (0,1),\sigma>0$, $\gamma  \in (0,2)$ and $\mathfrak{B}(x)$ is defined in \eqref{ball1}. Choose $x_{1} \in \mathfrak{C}$.\\
		\textbf{Step 1:} Compute
		\begin{align*}
			y_{n}=P_{\mathfrak{B}(x_{n})}(x_{n}-\lambda_{n}\mathcal{A}(x_{n})),
		\end{align*}
		where
		$\lambda_{n}=\sigma \delta^{k_{n}}$, where $k_{n}$ is the smallest non-negative integer such that 
		\begin{align} \label{line search}
			\lambda_{n}\|\mathcal{A}(x_{n})-\mathcal{A}(y_{n})\|\leq \mu \|x_{n}-y_{n}\|.
		\end{align}
		If $x_{n}=y_{n}$, then stop; and $y_{n}$ is a solution to the VIP, otherwise\\
		\textbf{Step 2:} Compute
		\begin{align*}
			d_{n}&=x_{n}-y_{n}-\lambda_{n}(\mathcal{A}(x_{n})-\mathcal{A}(y_{n})),\\ x_{n+1}&=P_{\mathfrak{B}(x_{n})}(x_{n}-\gamma \lambda_{n}\rho_{n}\mathcal{A}(y_{n})),
		\end{align*}
		where 
		\begin{align} \label{rhoeqn}
			\rho_{n}=\frac{\langle x_{n}-y_{n},d_{n}\rangle}{\|d_{n}\|^2}.
		\end{align}

		Set $n \gets n+1$ and go to \textbf{Step 1.}
		\\
		\begin{rmrk} \label{remarks1a}
			\begin{enumerate}
				\item[(a.)] It is easy to prove that $d_{n}=0$ if and only if $x_{n}=y_{n}$.
				\item[(b.)] The line search rule \eqref{line search} under the Assumptions (A2)-(A3) in Algorithm 3.1 is well defined \cite{lsearch}, and 
				\begin{align*}
					\min\{\sigma, \frac{\mu \delta}{L}\} \leq \lambda_{n} \leq \sigma.
				\end{align*}
				\item[(c.)] Since Slater's condition is satisfied for $\mathfrak{B}(x)$, for any $x \in \mathfrak{C}$, it follows that $r=\sqrt{\frac{1}{L_{f}^2}\|\nabla f(x_{n})\|^2-\frac{2}{L_{f}}f(x_{n})}>0$.
				\item[(d.)] Let $\{x_{n}\}$ and $\{y_{n}\}$ be two sequences generated by Algorithm 3.1. If we consider $\mathfrak{D}=\{x_{n}-\frac{1}{L_{f}}\nabla f(x_{n})\}$ and $r=\sqrt{\frac{1}{L_{f}^2}\|\nabla f(x_{n})\|^2-\frac{2}{L_{f}}f(x_{n})}$ in Lemma \ref{baush1}, we have
				\begin{align*}
					y_{n}=\begin{cases}
						x_{n}-\lambda_{n}\mathcal{A}(x_{n}), & \text{if } x_{n}-\lambda_{n}\mathcal{A}(x_{n}) \in \mathfrak{B}(x_{n}) \\
						x_{n}-\frac{1}{L_{f}}\nabla f(x_{n})+r\frac{\frac{1}{L_{f}}\nabla f(x_{n})-\lambda_{n}\mathcal{A}(x_{n})}{\|\frac{1}{L_{f}}\nabla f(x_{n})-\lambda_{n}\mathcal{A}(x_{n})\|}, & \text{otherwise.} 
					\end{cases}
				\end{align*}
				Similarly,
				\begin{align*}
					x_{n+1}=\begin{cases}
						x_{n}-\gamma \lambda_{n}\rho_{n}\mathcal{A}(y_{n}), & \text{if } x_{n}-\gamma \lambda_{n} \rho_{n}\mathcal{A}(y_{n}) \in \mathfrak{B}(x_{n}) \\
						x_{n}-\frac{1}{L_{f}}\nabla f(x_{n})+r\frac{\frac{1}{L_{f}}\nabla f(x_{n})-\gamma\lambda_{n} \rho_{n}\mathcal{A}(y_{n})}{\|\frac{1}{L_{f}}\nabla f(x_{n})-\gamma \lambda_{n} \rho_{n} \mathcal{A}(y_{n})\|}, & \text{otherwise.}
					\end{cases}
				\end{align*}
				
			\end{enumerate}
		\end{rmrk}
		We now present the convergence analysis.
		\begin{lemma} \label{lem3.1}
			Let $\{\rho_{n}\}$ be the sequence defined in Algorithm 3.1. Then, we have
			\begin{align*}
				\rho_{n}\geq \frac{1-\mu}{(1+\mu)^2}.
			\end{align*}
		\end{lemma}
		\begin{proof}
			Consider 
			\begin{align} \label{2eq}
				\langle x_{n}-y_{n},d_{n} \rangle &=\langle x_{n}-y_{n},x_{n}-y_{n}-\lambda_{n}(\mathcal{A}(x_{n})-\mathcal{A}(y_{n}))\rangle \notag \\
				&=\|x_{n}-y_{n}\|^2-\lambda_{n}\langle x_{n}-y_{n},\mathcal{A}(x_{n})-\mathcal{A}(y_{n})\rangle \notag \\
				& \geq \|x_{n}-y_{n}\|^2-\lambda_{n} \|x_{n}-y_{n}\|\|\mathcal{A}(x_{n})-\mathcal{A}(y_{n})\| ~~~[\text{Using \eqref{line search}}] \notag\\
				&\geq \|x_{n}-y_{n}\|^2-\mu \|x_{n}-y_{n}\|^2 \notag \\
				&=(1-\mu)\|x_{n}-y_{n}\|^2.
			\end{align}
			Also
			$\|d_{n}\|=\|x_{n}-y_{n}-\lambda_{n}(\mathcal{A}(x_{n})-\mathcal{A}(y_{n}))\|\leq \|x_{n}-y_{n}\|+\lambda_{n}\|\mathcal{A}(x_{n})-\mathcal{A}(y_{n})\|\leq (1+\mu)\|x_{n}-y_{n}\|.$ It implies that
			
			\begin{align} \label{3eq}
				\frac{1}{\|d_{n}\|^2} \geq \frac{1}{(1+\mu)^2\|x_{n}-y_{n}\|^2}.
			\end{align}
			From \eqref{2eq} and \eqref{3eq}, we get $\rho_{n}=\frac{\langle x_{n}-y_{n},d_{n}\rangle}{\|d_{n}\|^2}
			\geq \frac{(1-\mu)\|x_{n}-y_{n}\|^2}{(1+\mu)^2\|x_{n}-y_{n}\|^2}\geq \frac{1-\mu}{(1+\mu)^2}.$ Hence the result.
		\end{proof}
		\begin{lemma} \label{lem3.2}
			Let $\{x_{n}\}$ be a sequence generated by Algorithm 3.1. Then, under the Assumptions (A1)-(A4), we have
			\begin{align*}
				\|x_{n+1}-x^{*}\|^2 \leq \|x_{n}-x^{*}\|^2-\|x_{n}-x_{n+1}-\gamma \rho_{n}d_{n}\|^2-(2-\gamma)\gamma \rho_{n}^2\|d_{n}\|^2, ~\text{for all}~x^{*} \in VI(\mathfrak{C},\mathcal{A}).
			\end{align*}
		\end{lemma}
		\begin{proof}
			Consider $x^{*} \in VI(\mathfrak{C},\mathcal{A})$. Then, from Lemma \ref{le1} (ii), we have
			\begin{align} \label{5eq}
				\|x_{n+1}-x^{*}\|^2&=\|P_{\mathfrak{B}(x_{n})}(x_{n}-\gamma \lambda_{n}\rho_{n}\mathcal{A}(y_{n}))-x^{*}\|^2 \notag \\
				& \leq \|x_{n}-\gamma \lambda_{n}\rho_{n}\mathcal{A}(y_{n})-x^{*}\|^2-\|x_{n}-\gamma \lambda_{n}\rho_{n}\mathcal{A}(y_{n})-x_{n+1}\|^2 \notag \\
				&=\|x_{n}-x^{*}\|^2-\|x_{n+1}-x_{n}\|^2-2\langle \gamma \lambda_{n} \rho_{n}\mathcal{A}(y_{n}),x_{n+1}-x^{*} \rangle \notag \\
				&=\|x_{n}-x^{*}\|^2-\|x_{n+1}-x_{n}\|^2-2\gamma \lambda_{n}\rho_{n}\langle \mathcal{A}(y_{n}),x_{n+1}-x^{*}\rangle.
			\end{align}
			Since $\mathcal{A}$ is pseudo-monotone, we have $\langle \mathcal{A}(y_{n}),y_{n}-x^{*}\rangle \geq 0$. Now
			\begin{align} \label{6eq}
				\langle \mathcal{A}(y_{n}),x_{n+1}-x^{*}\rangle &=\langle \mathcal{A}(y_{n}),x_{n+1}-y_{n} \rangle+\langle \mathcal{A}(y_{n}),y_{n}-x^{*}\rangle \notag \\
				&\geq \langle \mathcal{A}(y_{n}),x_{n+1}-y_{n} \rangle.
			\end{align}
			Using Lemma \ref{le1} (i), we see that $\langle x_{n}-\lambda_{n}\mathcal{A}(x_{n})-y_{n},y_{n}-x_{n+1} \rangle\geq 0.$ It implies
			\begin{align*}
				\langle d_{n},x_{n+1}-y_{n}\rangle \leq \lambda_{n}\langle \mathcal{A}(y_{n}),x_{n+1}-y_{n}\rangle. 
			\end{align*}
			As $\gamma$ and $\rho_{n}>0$, we have
			\begin{align} \label{8eq}
				-2\gamma \rho_{n}\lambda_{n} \langle \mathcal{A}(y_{n}),x_{n+1}-y_{n} \rangle &\leq -2\gamma \rho_{n}\langle d_{n},x_{n+1}-y_{n} \rangle \notag \\
				&=-2\gamma \rho_{n}\langle x_{n}-y_{n},d_{n} \rangle+2\gamma \rho_{n} \langle d_{n},x_{n}-x_{n+1} \rangle.
			\end{align}
			From \eqref{6eq} and \eqref{8eq}, we see that
			\begin{align} \label{a8eq}
				-2\gamma \rho_{n}\lambda_{n}\langle \mathcal{A}(y_{n}),x_{n+1}-x^{*}\rangle \leq -2\gamma \rho_{n} \langle x_{n}-y_{n},d_{n} \rangle +2 \gamma \rho_{n} \langle d_{n},x_{n}-x_{n+1}\rangle.
			\end{align}
			From \eqref{rhoeqn}, we have
			\begin{align} \label{9eq}
				-2\gamma \rho_{n} \langle x_{n}-y_{n},d_{n} \rangle=-2\gamma \rho_{n}^2\|d_{n}\|^2.
			\end{align}
			Using \eqref{9eq} in \eqref{a8eq}, we have
			\begin{align} \label{a9eq}
				-2\gamma \rho_{n}\lambda_{n}\langle \mathcal{A}(y_{n}),x_{n+1}-x^{*}\rangle \leq -2\gamma \rho_{n}^2\|d_{n}\|^2+2\gamma \rho_{n}\langle d_{n},x_{n}-x_{n+1}\rangle.
			\end{align}
			Consider
			\begin{align*}
				\|x_{n}-x_{n+1}-\gamma \rho_{n}d_{n}\|^2=\|x_{n}-x_{n+1}\|^2+\gamma^2 \rho_{n}^2\|d_{n}\|^2-2\gamma \rho_{n}\langle d_{n},x_{n}-x_{n+1}\rangle.   
			\end{align*}
			On rearranging, we have
			\begin{align} \label{10eq}
				2\gamma \rho_{n}\langle d_{n},x_{n}-x_{n+1}\rangle=\|x_{n}-x_{n+1}\|^2+\gamma^2\rho_{n}^2\|d_{n}\|^2-\|x_{n}-x_{n+1}-\gamma \rho_{n}d_{n}\|^2.
			\end{align}
			Using \eqref{10eq} in \eqref{a9eq}, we get
			\begin{align} \label{a11eq}
				-2\gamma \rho_{n}\lambda_{n}\langle \mathcal{A}(y_{n}),x_{n+1}-x^{*}\rangle \leq -2\gamma \rho_{n}^2\|d_{n}\|^2+\|x_{n}-x_{n+1}\|^2+\gamma^2\rho_{n}^2\|d_{n}\|^2-\|x_{n}-x_{n+1}-\gamma \rho_{n}d_{n}\|^2.
			\end{align}
			Using \eqref{a11eq} in \eqref{5eq}, we see that
			\begin{align} \label{11eq}
				\|x_{n+1}-x^{*}\|^2 \leq \|x_{n}-x^{*}\|^2-\|x_{n}-x_{n+1}-\gamma \rho_{n}d_{n}\|^2-(2-\gamma)\gamma \rho_{n}^2\|d_{n}\|^2. 
			\end{align}
		\end{proof}
		\begin{theorem}
			Under the Assumptions (A1)-(A4), the sequence $\{x_{n}\}$ generated by Algorithm 3.1 converges to a solution of the VIP.
		\end{theorem}
		\begin{proof}
			From the definition of $d_{n}$, we have
			\begin{align} \label{12eq}
				\|d_{n}\|&=\|x_{n}-y_{n}-\lambda_{n}(\mathcal{A}(x_{n})-\mathcal{A}(y_{n}))\|\notag \\
				&\geq \|x_{n}-y_{n}\|-\lambda_{n}\|\mathcal{A}(x_{n})-\mathcal{A}(y_{n})\| \notag \\
				&\geq (1-\mu)\|x_{n}-y_{n}\|.
			\end{align}
			Using Lemmas \ref{lem3.1}, \ref{lem3.2} and \eqref{12eq}  in \eqref{11eq}, we get
			\begin{align} \label{13eq}
				\|x_{n+1}-x^{*}\|^2 &\leq \|x_{n}-x^{*}\|^2-\|x_{n}-x_{n+1}-\gamma \rho_{n}d_{n}\|^2-(2-\gamma)\gamma \rho_{n}^2\|d_{n}\|^2 \notag \\
				&\leq \|x_{n}-x^{*}\|^2-\|x_{n}-x_{n+1}-\gamma \rho_{n}d_{n}\|^2-(2-\gamma)\gamma \frac{(1-\mu)^2}{(1+\mu)^4}\|d_{n}\|^2 \notag \\
				& \leq \|x_{n}-x^{*}\|^2-\|x_{n}-x_{n+1}-\gamma \rho_{n}d_{n}\|^2-(2-\gamma)\gamma \frac{(1-\mu)^4}{(1+\mu)^4}\|x_{n}-y_{n}\|^2.
			\end{align}
			From \eqref{13eq}, we see $\|x_{n+1}-x^{*}\| \leq \|x_{n}-x^{*}\|.$ It implies that the sequence $\{\|x_{n}-x^{*}\|\}$ is decreasing and has a lower bound. Thus, it converges to some finite limit. Moreover, $\{x_{n}\}$ is a Fej\'{e}r monotone sequence with respect to $VI(\mathfrak{C},\mathcal{A})$ and thus is bounded. Now, we have to show that $\lim \limits_{n \to \infty}\|x_{n}-y_{n}\|=0$. From \eqref{13eq}, we see that
			\begin{align} \label{14eq}
				\sum \limits_{n=1}^{N}(2-\gamma)\gamma \frac{(1-\mu)^4}{(1+\mu)^4}\|x_{n}-y_{n}\|^2+\sum \limits_{n=1}^{N}\|x_{n}-x_{n+1}-\gamma \rho_{n}d_{n}\|^2 &\leq \sum \limits_{n=1}^{N}(\|x_{n}-x^{*}\|^2-\|x_{n+1}-x^{*}\|^2) \notag \\
				&=(\|x_{1}-x^{*}\|^2-\|x_{N+1}-x^{*}\|^2).
			\end{align}
			Taking $N \to \infty$ in \eqref{14eq}, we get
			\begin{align*}
				\sum \limits_{n=1}^{\infty}(2-\gamma)\gamma \frac{(1-\mu)^4}{(1+\mu)^4}\|x_{n}-y_{n}\|^2+\sum \limits_{n=1}^{\infty}\|x_{n}-x_{n+1}-\gamma \rho_{n}d_{n}\|^2\leq (\|x_{1}-x^{*}\|^2-L)<\infty,
			\end{align*}
			where $L=\lim \limits_{N \to \infty}\|x_{N+1}-x^{*}\|^2$. Thus, we have $\sum \limits_{n=1}^{\infty}\|x_{n}-y_{n}\|^2<\infty$, which further implies $\lim \limits_{n \to \infty}\|x_{n}-y_{n}\|=0$.
			\\
			Since $\{x_{n}\}$ is bounded, thus there exists a convergent subsequence $\{x_{n_{j}}\}$ such that $\lim \limits_{j \to \infty}x_{n_{j}}=\hat{x}.$
			Moreover, from Remark \ref{remarks1a}(b), we see that $\{\lambda_{n}\}$ has a converging subsequence $\{\lambda_{n_{j}}\}$ such that $\lim \limits_{j \to \infty}\lambda_{n_{j}}=\lambda$.
			Since $\{x_{n_{j}}\} \subseteq \mathfrak{C}$, we have $f(x_{n_{j}}) \leq 0$. From the continuity of $f$, we see that $f(\hat{x}) \leq 0$, thus $\hat{x}\in \mathfrak{C}$. Therefore, we have $r_{\hat{x}}>0$. 
			As $\{x_{n}\}$ is bounded, $\mathcal{A}$, $f$ and $\nabla f$ are continuous, we get that $\{\mathcal{A}(x_{n})\}$, $\{f(x_{n})\}$ and $\nabla f(x_{n})$ are also bounded. By explicitly expressing $y_{n}$, we see that $\{y_{n}\}$ is also bounded. Consider a sequence 
			\[
			\beta_{n} = 
			\begin{cases}
				1, & \text{if } x_{n} - \lambda_{n} A x_{n} \in \mathfrak{B}(x_{n}) \\
				\frac{
					\sqrt{
						\frac{1}{L_{f}^2} \|\nabla f(x_{n})\|^2 - \frac{2}{L_{f}} f(x_{n})
					}
				}{
					\left\| \frac{1}{L_{f}} \nabla f(x_{n}) - \lambda_{n} \mathcal{A}(x_{n}) \right\|
				}, & \text{otherwise}.
			\end{cases}
			\]
			From \eqref{ball1}, it is easy to see that $0<\beta_{n} \leq 1$, thus without loss of generality, we assume that a subsequence $\{\beta_{n_{j}}\}$ of $\{\beta_{n}\}$ converges to $\beta$. From the definition of $\beta_{n}$, we claim that $0<\beta \leq 1$. As
			\begin{align*}
				\beta = \lim \limits_{j \to \infty} \beta_{n_j} 
				= \lim \limits_{j \to \infty} 
				\frac{
					\sqrt{
						\frac{1}{L_g^2} \|\nabla f(x_{n_j})\|^2 - \frac{2}{L_g} f(x_{n_j})
					}
				}{
					\left\| \frac{1}{L_g} \nabla f(x_{n_j}) - \lambda_{n_j} \mathcal{A}(x_{n_j}) \right\|
				} > 0.
			\end{align*}
			Therefore, we have $0<\beta \leq 1$. We now prove that $\hat{x} \in VI(\mathfrak{C},\mathcal{A})$. We discuss the following two cases.\\
			\textbf{Case 1:} There exists infinite terms $\{x_{n_{j}}\}$ such that $x_{n_{j}}-\lambda_{n_{j}}\mathcal{A}(x_{n_{j}}) \in \mathfrak{B}(x_{n_{j}})$. Then, from the explicit expression of $y_{n}$, we have
			\begin{align} \label{14e1}
				y_{n_{j}}=x_{n_{j}}-\lambda_{n_{j}}\mathcal{A}(x_{n_{j}}).
			\end{align}
			Taking $j \to \infty$ on both sides of \eqref{14e1}, we see that $\mathcal{A}(\hat{x})=0$, which further implies that $\hat{x} \in VI(\mathfrak{C},\mathcal{A})$.\\
			\textbf{Case 2:} There exists $n_{0} \in \mathbb{N}$ such that $x_{n_{j}}-\lambda_{n_{j}}\mathcal{A}(x_{n_{j}}) \notin \mathfrak{B}(x_{n_{j}})$, for any $n_{j}>n_{0}$. Now, $y_{n}=P_{\mathfrak{B}(x_{n})}(x_{n}-\lambda_{n}\mathcal{A}(x_{n})) \in \mathfrak{B}(x_{n})$. It implies that $y_{n} \in \partial \mathfrak{B}(x_{n}).$ Therefore, by the definition of the moving ball, we have
			\begin{align} \label{15e1}
				f(x_{n_{j}})+\langle \nabla f(x_{n_{j}}),y_{n_{j}}-x_{n_{j}}\rangle+\frac{L_{f}}{2}\|y_{n_{j}}-x_{n_{j}}\|^2=0.
			\end{align}
			Taking $j \to \infty$ on both sides of \eqref{15e1}, we get $f(\hat{x})=0$, which further implies $\hat{x} \in \partial \mathfrak{C}$. Using the explicit form of $y_{n}$, we see that
			\begin{align} \label{16e1}
				y_{n_{j}}=x_{n_{j}}-\beta_{n_{j}}\lambda_{n_{j}}\mathcal{A}(x_{n_{j}})-\frac{(1-\beta_{n_{j}})}{L_{f}}\nabla f(x_{n_{j}}).
			\end{align}
			Taking $j \to \infty$ in \eqref{16e1}, we see that
			\begin{align*}
				\beta \lambda \mathcal{A}(\hat{x})+\frac{(1-\beta)}{L_{f}}\nabla f(\hat{x})=0.
			\end{align*}
			Therefore, we obtain 
			\begin{align*}
				\mathcal{A}(\hat{x})=\frac{-(1-\beta)}{\beta \lambda L_{f}}\nabla f(\hat{x}).
			\end{align*}
			If $\beta=1$, then we have $\mathcal{A}(\hat{x})=0$, which implies that $\hat{x} \in VI(\mathfrak{C},\mathcal{A})$. If $0<\beta <1$, then we have $\frac{(1-\beta)}{\beta \lambda L_{f}}>0$, thus from Lemma \ref{2.5lem}, we conclude that $\hat{x} \in VI(\mathfrak{C},\mathcal{A})$. 
			Finally, we establish that the sequence $\{x_{n}\}$ converges to $\hat{x}$. We have already demonstrated that the sequence $\{\|x_{n}-\hat{x}\|\}$ is monotonically decreasing and convergent. Hence
			\begin{align*}
				\lim \limits_{n \to \infty}\|x_{n}-\hat{x}\|=\lim \limits_{j \to \infty}\|x_{n_{j}}-\hat{x}\|=0,
			\end{align*}
			which means $\{x_{n}\}$ converges to $\hat{x}$. Hence the result.
			
		\end{proof}

		\section{Numerical Experiments}
		This section provides some numerical experiments to validate our main result. We compare our algorithm with well-established algorithms of Cao and Guo \cite{4}, Feng et al. \cite{7} and Zhang et al. \cite{6}. 
		
		We use $E_{n}=\|x_{n}-y_{n}\|$ to calculate the $n$-th iteration error. The convergence of $E_{n} \to 0$ implies that the sequence $\{x_{n}\}$ converges to the solution of the variational inequality problem. The parameters associated with the different algorithms are given as follows:
		\begin{itemize}
			\item \textbf{Algorithm 3.1}: $\mu=0.01$, $\delta=0.0005$, $\sigma=7$, and $\gamma=0.99$.
			\item \textbf{Cao and Guo}: $\tau=0.0018$, and $\alpha_{n}=\frac{1}{4n+1}$.
			\item \textbf{Feng et al.}: $0< \tau < \frac{1}{L_{\mathcal{A}}}$
			\item \textbf{Zhang et al.}: $0 < \tau_{1} \leq \frac{1}{4L_{\mathcal{A}}}$
		\end{itemize}
		\begin{example} \label{example1}
			Consider the mapping $\mathcal{A}:\mathbb{R}^{4} \to \mathbb{R}^{4}$ defined as
			\begin{align*}
				\mathcal{A}(x_{1},x_{2},x_{3},x_{4})=\begin{bmatrix}
					3x_{1}^2+2x_{1}x_{2}+2x_{2}^2+x_{3}+3x_{4}-6 \\
					2x_{1}^2+x_{1}+x_{2}^2+10x_{3}+2x_{4}-2 \\
					3x_{1}^2+x_{1}x_{2}+2x_{2}^2+2x_{3}+9x_{4}-9 \\
					x_{1}^2+3x_{2}^2+2x_{3}+3x_{4}-3 
				\end{bmatrix} .
			\end{align*}
			The feasible set $\mathfrak{C} \subseteq \mathbb{R}^{4}$ is an ellipsoid defined as 
			\begin{align*}
				\mathfrak{C}=\{x \in \mathbb{R}^4: (x-t)'T(x-t) \leq u^2\},
			\end{align*}
			where $T$ is a positive semidefinite matrix, $t\in \mathbb{R}^4$ and $u>0$. It is known that $\mathcal{A}$ is pseudo-monotone and Lipschitz continuous \cite{12}. Define $f: \mathbb{R}^4 \to \mathbb{R}$ by 
			\begin{align*}
				f(x)=\frac{1}{2}\big[(x-t)'T(x-t)-u^2\big].
			\end{align*}
			Then $\mathfrak{C}$ can be rewritten as $\mathfrak{C}=\{x \in \mathbb{R}^4: f(x) \leq 0\}$. The gradient of $f(x)$ is $\nabla f(x)=T(x-t)$, for all $x \in \mathbb{R}^{4}$. Moreover
			\begin{align*}
				\|\nabla f(x)-\nabla f(y)\|=\|T(x-t)-T(y-t)\| \leq \|T\|\|x-y\|.
			\end{align*}
			It implies that $\nabla f$ is Lipschitz continuous with constant $\|T\|$. Furthermore, Slater's condition holds for $\mathfrak{C}$, i.e. $f(t)=-\frac{1}{2}u^2<0$. Since $\mathfrak{C}$ is an ellipsoid, computing its projections is not straightforward. So, we project onto a moving ball such that the projections always remain under the feasible set $\mathfrak{C}$. For our experiment, we randomly generate the positive semidefinite matrix $T$, vector $t$, and $u$. We choose the starting point as $x_{1}=t$. We include a stopping criterion given by $E_{n} \leq 10^{-15}$. The consequences of our algorithm can be seen in Figure \ref{ex1}. 
			
			\vspace{1.1cm}
			\begin{figure}[ht!]
				\centering
				\includegraphics[scale=0.5,trim=40mm 40mm 40mm 30mm]{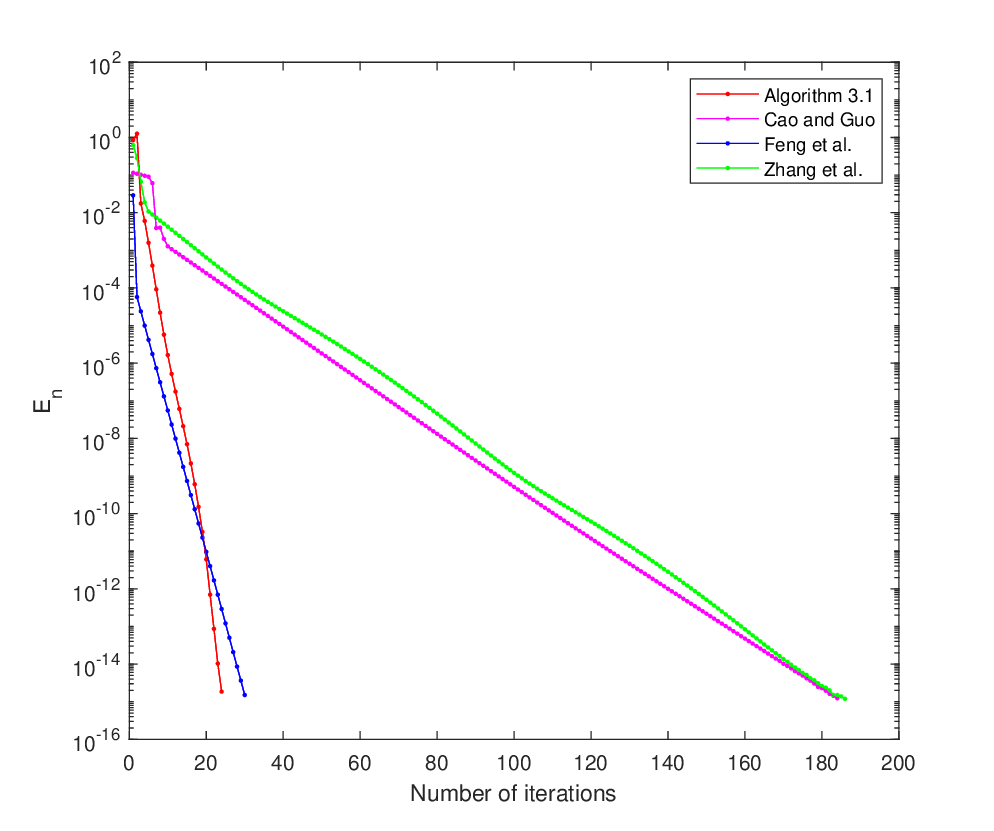}  
				\vspace{0.6in}
				\caption{Comparison of algorithms for Example \ref{example1} }
				\label{ex1}
			\end{figure}
			\vspace{0.5cm}
		\end{example}

		\begin{example} \label{example2}
			Consider the mapping $\mathcal{A}: \mathbb{R}^{n} \to \mathbb{R}^{n}$ defined as
			\begin{align*}
				\mathcal{A}(x)=G(x)+Mx+e,
			\end{align*}
			where $G(x)_{i}=\arctan x_{i}, ~i=\{1,2,...,n\}$,
			\begin{align*}
				M=
				\begin{bmatrix}
					4 & -2 & 0 & \cdots & 0 & 0 \\
					1 & 4 & -2 & \cdots & 0 & 0 \\
					0 & 1 & 4 & \cdots & 0 & 0 \\
					\vdots & \vdots & \vdots & \ddots & \vdots & \vdots \\
					0 & 0 & 0 & \cdots & 4 & -2 \\
					0 & 0 & 0 & \cdots & 1 & 4 \\
				\end{bmatrix},
			\end{align*}
			and $e=(-1,-1,-1,...,-1)'$. The feasible set $\mathfrak{C}$ is defined as $\mathfrak{C}=\{x \in \mathbb{R}^{n}: (x-t)'T(x-t) \leq u^2\}$. Since $\mathcal{A}$ is pseudo-monotone and Lipschitz continuous, we apply our algorithm to find the solution. We randomly generate the positive semidefinite matrix $T$, vector $t$, and $u$. The starting point $x_{1}$ is considered to be $t$. We include a stopping criterion given by $E_{n} \leq 10^{-10}$. The consequences of our algorithm can be seen in Figure \ref{ex2}.
			\vspace{1.1cm}
			\begin{figure}[ht!]
				\centering
				\includegraphics[scale=0.5,trim=40mm 40mm 40mm 30mm]{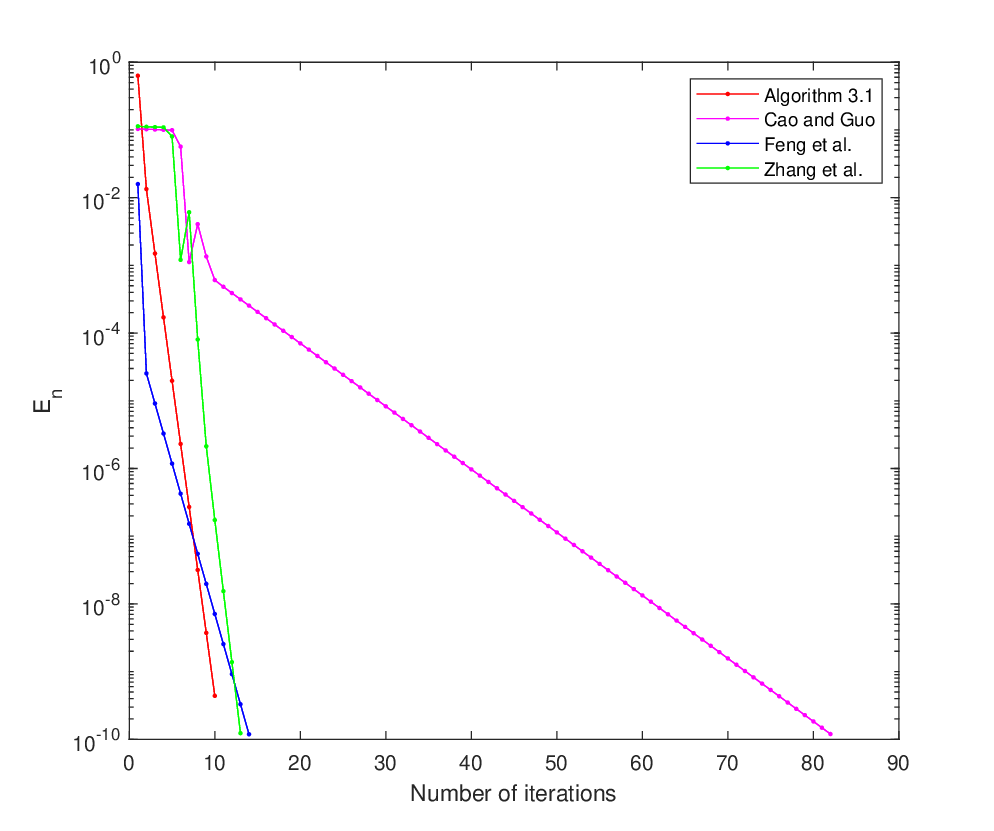}  
				\vspace{0.6in}
				\caption{Comparison of algorithms for Example \ref{example2} when $n=1000$ }
				\label{ex2}
			\end{figure}
			\vspace{0.5cm}

		\end{example}
		
		\section{Conclusion}
		In this paper, we used a moving ball algorithm along with a line search rule to solve the variational inequality problem. Taking inspiration from the extragradient method and the projection-contraction method, we proved a convergence result for our algorithm that is easy to implement in cases where the structure of the convex set (feasible set) is complicated. Since the projection onto the ball has an explicit expression, the applicability of our algorithm is quite easy. Moreover, our algorithm does not need any additional information on the pseudo-monotone cost operator other than being Lipschitz continuous. This means that we do not need to prove Condition \textbf{(T)} on the cost operator. We discussed the effectiveness of our algorithm over other well-established algorithms. It has been shown that our algorithm performs better compared to earlier established algorithms. For future work, we try to generalize the concept in the setting of Banach spaces using some inertial techniques.
		\section*{Declarations}
		
		\textbf{Ethical Approval.} Not Applicable as no both human and/ or animal studies. 
		
		\textbf{Data Availability.} No underlying data was collected or produced in this study.

		\textbf{Competing interests.}
		The authors declare that there is no competing interest in the publication of this paper.
		
		\textbf{Funding.} There is no funding.
		
		\textbf{Authors' contributions.} All authors contribute equally.
		
		\textbf{Acknowledgements} 
		The first author is grateful to CSIR, New Delhi, India, for providing a senior research fellowship (File 09/0677(13166)/2022-EMR-I).


\begin{thebibliography}{99}
			\addcontentsline{toc}{chapter}{BIBLIOGRAPHY}
			\bibitem{10} A. Auslender, R. Shefi and M. Teboulle, A moving balls approximation method for a class of
			smooth constrained minimization problems. SIAM J. Optim. 20 (2010) 3232–3259.
			\bibitem{8} H. H. Bauschke and P. L. Combettes, Convex Analysis and Monotone Operator Theory in Hilbert
			Spaces. Springer, New York (2011).
			\bibitem{4} Y. Cao and K. Guo, On the convergence of inertial two-subgradient extragradient method for solving variational inequality problems. Optim. 69 (2020) 1237–1253.
			\bibitem{3} Y. Censor, A. Gibali and S. Reich, The subgradient extragradient method for solving variational inequalities in Hilbert space. J. Optim. Theory Appl. 148 (2011) 318–335.
			\bibitem{5}J. X. Chen and  M. L. Ye, A new modified two-subgradient extragradient algorithm for solving variational
			inequality problems. J. Math. Res. Appl. 42 (2022) 402–412. 
			\bibitem{7}L. Feng, Y.Zhang and Y. He, A new feasible moving ball projection algorithm for pseudomonotone variational inequality. Optim. Lett. 18 (2024) 1437-1455.
			\bibitem{1} A. A. Goldstein, Convex programming in Hilbert space. Bull. Amer. Math. Soc. 70 (1964) 709–710.
			\bibitem{11} S.N. He and H.K. Xu, Uniqueness of supporting hyperplanes and an alternative to solutions of variational
			inequalities. J. Global Optim. 57 (2013), 1375–1384.
			\bibitem{2} G. M. Korpelevich, The extragradient method for finding saddle points and other problems. EKON MAT METOD. 12 (1976) 747–756.
			\bibitem{Wang} B. Ma and W. Wang, Self-adaptive subgradient extragradient-type methods for solving variational inequalities. J. Inequal. Appl. 54 (2022).
			\bibitem{9} J. M. Ortega and W. C. Rheinboldt, Iterative Solution of Nonlinear Equations in Several Variables. Classics in Applied Mathematics, Vol. 30. Philadelphia: Society
			for Industrial and Applied Mathematics (2000).
			\bibitem{12} M.V. Solodov and B.F. Svaiter, A new projection method for variational inequality problems. SIAM J. 
			Control. Optim. 37 (1999) 765–776. 
			\bibitem{1a} G. Stampacchia, Formes bilineaires coercitives sur les ensembles convexes. C.R. Acad. Sci. Paris.
			258 (1964) 4413-4416.
			\bibitem{lsearch}  D.V. Thong and  D.V. Hieu, Weak and strong convergence theorems for variational inequality problems.
			Numer. Algorithms 78 (2018) 1045–1060.
			
			\bibitem{6} Y. Zhang, L. Feng and Y. He, A new low-cost feasible projection algorithm for pseudomonotone variational inequalities. Numer. Algorithms 94 (2023) 1031-1054.
			
			
			
			
			
			
			
			
			
		\end{thebibliography}
	\end{document}